\documentclass[a4paper,12pt]{amsart}
\usepackage{graphics}
\usepackage{graphicx}
\usepackage{amssymb}
\usepackage{amsmath}
\usepackage{amsthm}
\usepackage{multicol}
\usepackage[english]{babel}
\usepackage{cmap}
\usepackage{setspace}
\newcommand{\Q}{\mathbb{Q}}

\usepackage{setspace}
\makeatletter
 
  \@addtoreset{equation}{section}
\makeatother
\makeatletter
\def\tbcaption{\def\@captype{table}\caption}
\makeatother
\newtheorem{theorem}{Theorem}[section]
\newtheorem{lemma}[theorem]{Lemma}
\newtheorem{cor}[theorem]{Corollary}
\theoremstyle{definition}

\theoremstyle{remark}
\newtheorem{remark}[theorem]{Remark}

\numberwithin{equation}{section}

\newcommand{\beq}[1]{\begin{equation}\label{#1}}
\newcommand{\eeq}{\end{equation}}

\newcommand{\bZ}{\ensuremath{\mathbb{Z}}}
\newcommand{\bQ}{\ensuremath{\mathbb{Q}}}

\newcommand{\bL}{\ensuremath{\mathbb{L}}}

\newcommand{\abs}[1]{\left\vert#1\right\vert}

\begin{document}
\title[$F_n$ or $L_n$ numbers as products of three repdgits in base $g$]{\small Fibonacci and Lucas numbers as products of three repdgits in base $g$}

\author[K. N. Ad\'edji, A. Filipin and A. Togb\'e]{Kou\`essi Norbert Ad\'edji, Alan Filipin and Alain Togb\'e}

\date{\today}
\maketitle

\begin{abstract}
Recall that repdigit in base $g$ is a positive integer that has only one digit in its base $g$ expansion, i.e. a number of the form $a(g^m-1)/(g-1)$, for some positive integers $m\geq 1$, $g\geq 2$ and $1\leq a\leq g-1$.  In the present study we investigate all Fibonacci or Lucas numbers  which are expressed as products of three repdigits in base $g$. As illustration, we consider the case $g=10$ where we show that the numbers $144$ and $18$ are the largest Fibonacci and Lucas numbers which can be expressible as products of three repdigits respectively. All this can be done using linear forms in logarithms of algebraic numbers.
\end{abstract}
\maketitle




\section{Introduction}\label{intr}

Let $\{F_n\}_{n\ge 0}$ be  the Fibonacci sequence given by $F_{n+2}=F_{n+1} + F_n,$ with initial values $F_0=0$ and $F_1=1$ and let $\{L_n\}_{n\ge 0}$ be the Lucas sequence defined by $L_{n+2}=L_{n+1} + L_n,$ where $L_0 =2$ and $L_1=1.$ If 
\[
(\alpha, \beta) =\left(\dfrac{1+\sqrt{5}}{2},\dfrac{1-\sqrt{5}}{2} \right)
\]
is the pair of roots of the characteristic equation $x^2 -x-1=0$ of both the Fibonacci and Lucas numbers, then  Binet's formulas for their general terms are
\begin{align}\label{FLn}
F_n=\dfrac{\alpha^n-\beta^n}{\alpha-\beta}\quad \text{and}\quad L_n=\alpha^n+\beta^n,\quad \text{for}\quad n\ge 0.
\end{align}
It can be seen that $1<\alpha <2,\;-1<\beta <0$ and $\alpha\beta=-1.$ The following relations between $n$-th Fibonacci number $F_n,$ $n$-th Lucas number $L_n$ and $\alpha$ are well known
\begin{align}\label{F_n-L_n}
\alpha^{n-2} \le F_n\le \alpha^{n-1}\quad\text{and}\quad \alpha^{n-1} \le L_n\le 2\alpha^{n},\quad\text{for}\quad n\ge 0.
\end{align}
Now, we are going to introduce the second concept of this study related to repdigits. Let $g\ge 2$ be an integer. A positive integer $N$ is called a repdigit in base $g$ or simply a $g$-repdigit if all of the digits in its base $g$ expansion are equal. Indeed, $N$ is of the form
\[
a\left(\dfrac{g^m-1}{g-1}\right),\quad\text{for}\; m\ge1,\;  a\in \{1,2,\ldots,g-1\}.
\]
Taking $g=10,$ the positive integer $N$ is simply called repdigit. The study of Diophantine equations involving linear recurrent sequences and repdigits has been considered in recent years by many theorists. First, Luca showed in \cite{Luca:2000} that the number $55$ is the largest repdigit in the Fibonacci sequence. After this, many authors have worked on other similar problems (see \cite{ALT:2021}, \cite{Ddamulira:2020}, \cite{EKS:2019}, \cite{EKS:2021} and references therein). In \cite{EK:2019}, the authors solved the problem of ﬁnding the Fibonacci or Lucas numbers as products of two repdigits. In \cite{EKS2:2021}, the authors tackled the problem of ﬁnding Fibonacci or Lucas numbers which are products of two repdigits in base $b$. In \cite{Adedji-Filipin-Togbe:2020-rim}, we found all padovan and perrin numbers which are products of two repdigits in base $b$ with $2\le b\le 10$. This study deals with Fibonacci and Lucas numbers which are products of three repdigits in base $g$. In other words, we study the Diophantine equations
\begin{align}\label{main_equation1}  
F_k=d_1 \dfrac{g^\ell-1}{g-1}\cdot d_2\dfrac{g^m-1}{g-1}\cdot d_3\dfrac{g^n-1}{g-1},
\end{align}
and
\begin{align}\label{main_equation2} 
L_k=d_1 \dfrac{g^\ell-1}{g-1}\cdot d_2\dfrac{g^m-1}{g-1}\cdot d_3\dfrac{g^n-1}{g-1},
\end{align}
where $d_1, d_2, d_3, k, \ell, m$ and $n$ are positive integers such that 

$1 \le d_1,d_2,d_3 \le g-1$ and $g\ge 2$ with $n\ge 1$, $\ell\le m\le n$.
The novelty to the present work lies in effectiveness, in the sense that $k$ and $n$ can be effectively bounded in terms of $g.$ This may be obtained using Baker's method. There are several different estimates of Baker-type lower bounds for linear forms in logarithms. In this study, we use the most common Baker type method due to Matveev \cite{Matveev:2000} or \cite[Theorem~9.4]{BMS:2006}. Thus, our main results are as follows.

\begin{theorem}\label{main_result1}
Let $g\ge 2$ be an integer. Then the Diophantine equation \eqref{main_equation1}  has only finitely many solutions in integers  $k, d_1, d_2, d_3, g, \ell, m, n$ such that $n\ge 1.$ Namely, we have 
$$
\ell \le m\le n<1.08\times 10^{48}\log^9 g\quad\text{and}\quad k<1.08\times 10^{49}\log^{10}g.
$$
\end{theorem}

\begin{theorem}\label{main_result2}
Let $g\ge 2$ be an integer. Then the Diophantine equation \eqref{main_equation2}  has only finitely many solutions in integers  $k, d_1, d_2, d_3, g, \ell, m, n$ such that $n\ge 1.$ Namely, we have 
$$
\ell \le m\le n<7.73\times 10^{47}\log^9 g\quad\text{and}\quad k<7.73\times 10^{48}\log^{10}g.
$$
\end{theorem}

The organization of this paper is as follows. In Section~\ref{prelim} we will cite the results that we will use in Sections~\ref{Sec_Fibo} and \ref{Sec_Lucas}, where our fundamental results of this paper will be fully proven. Also, we devote Section~\ref{The end} to some concluding remarks.

\section{Auxiliary results}\label{prelim}

We begin this section with a few reminders about logarithmic height of an algebraic number. Let $\eta$ be an algebraic number of degree $d,$ let $a_0 >0$ be the leading coefficient of its minimal polynomial over $\bZ$ and let $\eta=\eta^{(1)},\ldots,\eta^{(d)}$ denote its conjugates. The quantity defined by  
\[
 h(\eta)= \frac{1}{d}\left(\log |a_0|+\sum_{j=1}^{d}\log\max\left(1,\left|\eta^{(j)} \right| \right) \right)
\]
is called the logarithmic height of $\eta.$ Some properties of height are as follows. For $\eta_1, \eta_2$ algebraic numbers and $m\in \bZ,$ we have
\begin{align*}
 h(\eta_1 \pm \eta_2) &\leq h(\eta_1)+ h(\eta_2) +\log2,\\
h(\eta_1\eta_2^{\pm}) &\leq h(\eta_1) + h(\eta_2),\\
h(\eta_1^m)&=|m|h(\eta_1).
\end{align*}
If $\eta=\dfrac{p}{q}\in\bQ$ is a rational number in reduced form with $q>0,$ then the above definition reduces to $h(\eta)=\log(\max\{|p|,q\}).$ We can now present the famous Matveev result used in this study. Thus, let $\bL$ be a real number field of degree $d_{\bL}$, $\eta_1,\ldots,\eta_s \in \bL$ and $b_1,\ldots,b_s \in \bZ  \setminus\{0\}.$ Let $B\ge \max\{|b_1|,\ldots,|b_s|\}$ and
\[
\Lambda=\eta_1^{b_1}\cdots\eta_s^{b_s}-1.
\]
Let $A_1,\ldots,A_s$ be real numbers with 
\[
A_i\ge \max\{d_{\bL} h(\eta_i), |\log\eta_i|, 0.16\},\quad i=1,2,\ldots,s.
\]
With the above notations, Matveev proved the following result.
\begin{theorem}\label{Matveev}
Assume that $\Lambda\neq 0.$ Then
\[
\log|\Lambda|>-1.4\cdot30^{s+3}\cdot s^{4.5}\cdot d_{\bL}^2 \cdot(1+\log d_{\bL})\cdot(1 +\log B)\cdot A_1\cdots A_s.
\]
\end{theorem}

The following lemma will also be used in order to prove our subsequent results.
\begin{lemma}[Lemma~7 of \cite{SGL}]\label{lemma_G-Luca}
If $l\ge 1,\; H>\left(4l^2\right)^l$ and $H>L/(\log L)^l,$ then 
$$ 
L<2^l H(\log H)^l.
$$
\end{lemma}

The upper bounds of the variables of equations \eqref{main_equation1} and \eqref{main_equation2} obtained after the application of Theorem~\ref{Matveev} are very large for a very fast search for solutions by a computer program. To overcome this situation, a reduction of the upper bounds is necessary. For this reduction purpose, we present a variant of the reduction method of Baker and Davenport due to Dujella and Peth\H o \cite{Dujella-Peto}. For a real number $x$, we write $||x||:= \min\{|x - n|: n\in \bZ\}$ for the distance from $x$ to the nearest integer.
\begin{lemma}[Lemma~5a of \cite{Dujella-Peto}]\label{Dujella-Peto}
	Let $M$  be a positive integer, let $p/q$ be a convergent of the continued fraction of the irrational $\tau$ such that $q> 6M$, and let $A,B,\mu$ be some real numbers with $A> 0$ and $B> 1$. Let
	$$
	\varepsilon=||\mu q||-M\cdot||\tau q||,
	$$
	where $||\cdot||$ denotes the distance from the nearest integer. If    $\varepsilon>0$, then there is no solution of the inequality
	$$
	0 <\abs{m\tau-n+\mu}<AB^{-w}
	$$
	in positive integers $m,n$ and $k$ with 
	$$
	m\leq M \quad \text{and} \quad w\geq\dfrac{\log(Aq/\varepsilon)}{\log B}.
	$$ 
\end{lemma}
Note that Lemma~\ref{Dujella-Peto} cannot be applied when $\mu = 0$ (since then $\varepsilon< 0$) or when $\mu$ is a multiple of $\tau$. For this
case, we use the following well known technical result from Diophantine approximation, known as Legendre's criterion. 

\begin{lemma}[Legendre \cite{Cohen:1993}]\label{Lemma-Legendre}
Let $\kappa$ be a real number and $x, y$ integers such that $$  
\left|\kappa - \frac{x}{y} \right|<\frac{1}{2y^2}.
$$
Then $x/y = p_k /q_k$ is a convergent of $\kappa$. Further, let $M$ and $N$ be a nonnegative integers such that $q_N>M.$ Then putting $a(M):=\max\{a_i: i=0,1,2,\ldots,N\},$ the inequality
 $$
 \left|\kappa - \frac{x}{y} \right|\geq\frac{1}{(a(M)+2)y^2},
 $$
holds for all pairs $(x, y)$ of positive integers with $0<y<M.$
\end{lemma}

\section{Fibonacci numbers as products of three repdgits in base $g$}\label{Sec_Fibo}

Our first aim is to prove Theorem~\ref{main_result1}.

\subsection{Proof of Theorem~\ref{main_result1}}

Note that if $n=1,$ then $\ell=m=1$ and therefore the Diophantine equation \eqref{main_equation1} becomes $F_k=d_1d_2d_3.$ Combining this with \eqref{F_n-L_n}, we have $k\le 3\log (g-1)/\log\alpha+2$. So, in this case the bound of $k$ from Theorem~\ref{main_result1} easily holds. For the rest of the proof we consider $n\ge 2.$ The following result will be useful in proving our main result which gives a
relation between $n$ and $k$ in equation \eqref{main_equation1}.
\begin{lemma}\label{k}
All solutions of the Diophantine equation \eqref{main_equation1} satisfy
$$
k<3n\dfrac{\log g}{\log \alpha}+2<10n\log g.
$$
\end{lemma}
\begin{proof}
From \eqref{F_n-L_n}, we have
$$  
\alpha^{k-2}\leq F_k=d_1\frac{g^\ell-1}{g-1}\cdot d_2\frac{g^m-1}{g-1}\cdot d_3\frac{g^n-1}{g-1}\le  (g^n-1)^3<g^{3n}.
$$
Taking logarithm on both sides, we get $(k-2)\log\alpha<3n\log g.$ Since, $n\ge 2$ and $g\ge 2,$ we obtain the desired inequalities. This ends the proof. 
\end{proof}

To start, we find the upper bounds for the variables $n, \ell, m$ of equation \eqref{main_equation1}. Using \eqref{FLn} and \eqref{main_equation1}, we get 
$$  
F_k=\dfrac{\alpha^k}{\sqrt{5}}-\dfrac{\beta^k}{\sqrt{5}}=d_1\frac{g^\ell-1}{g-1}\cdot d_2\frac{g^m-1}{g-1}\cdot d_3\frac{g^n-1}{g-1}
$$
to obtain 
\begin{align}\label{equation11}
\dfrac{\alpha^k}{\sqrt{5}}-\frac{d_1d_2d_3g^{\ell+m+n}}{(g-1)^3}&=\dfrac{\beta^k}{\sqrt{5}}-\frac{d_1d_2d_3g^{\ell+m}}{(g-1)^3}-\frac{d_1d_2d_3g^{n+\ell}}{(g-1)^3}+\frac{d_1d_2d_3g^\ell}{(g-1)^3}\\
&\nonumber -\frac{d_1d_2d_3g^{n+m}}{(g-1)^3}+\frac{d_1d_2d_3g^m}{(g-1)^3}+\frac{d_1d_2d_3g^n}{(g-1)^3}-\frac{d_1d_2d_3}{(g-1)^3}.
\end{align}
Taking the absolute values of both sides of \eqref{equation11}, we get 
\begin{align}\label{equation12}
\left| \dfrac{\alpha^k}{\sqrt{5}}-\frac{d_1d_2d_3g^{\ell+m+n}}{(g-1)^3}\right|&\le \dfrac{1}{\alpha^k\sqrt{5}}+\frac{d_1d_2d_3g^{\ell+m}}{(g-1)^3}+\frac{d_1d_2d_3g^{n+\ell}}{(g-1)^3}+\frac{d_1d_2d_3g^\ell}{(g-1)^3}\\
&\nonumber +\frac{d_1d_2d_3g^{n+m}}{(g-1)^3}+\frac{d_1d_2d_3g^m}{(g-1)^3}+\frac{d_1d_2d_3g^n}{(g-1)^3}+\frac{d_1d_2d_3}{(g-1)^3}.
\end{align}
Dividing both sides of \eqref{equation12} by $\frac{d_1d_2d_3g^{\ell+n+m}}{(g-1)^3}$ and using the fact that $\ell\le m\le n$ gives us the following inequalities.
\begin{align}
\nonumber \left|\frac{(g-1)^3\cdot\alpha^k\cdot g^{-(\ell+n+m)}}{d_1d_2d_3\sqrt{5}}-1 \right|& \leq \frac{(g-1)^3}{\alpha^k d_1d_2d_3g^{\ell+n+m}\sqrt{5}}+\frac{1}{g^n}+\frac{1}{g^m}+\frac{1}{g^{n+m}}\\
\notag &+\frac{1}{g^\ell}+\frac{1}{g^{n+\ell}}+\frac{1}{g^{\ell+m}}+\dfrac{1}{g^{\ell+m+n}}\\
\notag &< 8\cdot g^{-\ell}.
\end{align}
From this, it follows that 
\begin{equation}\label{equation13}
\left|\frac{(g-1)^3}{d_1d_2d_3\sqrt{5}}\cdot\alpha^k\cdot g^{-(\ell+n+m)}-1 \right| < \frac{8}{g^\ell}.
\end{equation}
Put
$$
\Lambda_1:=\frac{(g-1)^3}{d_1d_2d_3\sqrt{5}}\cdot\alpha^k\cdot g^{-(\ell+n+m)}-1.
$$
We need to show $\Lambda_1\ne 0$. Suppose $\Lambda_1=0$, then
$$
\alpha^{2k}=\dfrac{5(d_1d_2d_3)^2}{(g-1)^6}\cdot g^{2(\ell+m+n)}\in \Q(\sqrt{5})
$$
which is false. Thus $\Lambda_1\ne 0.$ To apply Theorem~\ref{Matveev} to $\Lambda_1,$ we choose the following data
$$
(\eta_1, b_1):=\left(\frac{(g-1)^3}{d_1d_2d_3\sqrt{5}},1 \right),\; (\eta_2, b_2):=(\alpha, k),\; (\eta_3, b_3):=(g, -(\ell+m+n))
$$
and $s:=3.$ Note that $\eta_1, \eta_2, \eta_3\in \bQ(\alpha)$ and $b_1, b_2, b_3\in \bZ.$ The degree $d_\bL=[\bQ(\alpha):\bQ]$ is $2$, where $\bL$ is $\bQ(\alpha).$ According to the inequalities from Lemma~\ref{k}, we can take $B:=10n\log g$ because $B\ge \max\{|b_1|,|b_2|,|b_3|\}.$ To estimate the parameters $A_1, A_2, A_3 ,$ we calculate the logarithmic heights
of $\eta_1, \eta_2, \eta_3$ as follows:
\begin{align*}
h(\eta_1)=h\left(\frac{(g-1)^3}{d_1d_2d_3\sqrt{5}} \right)&\le h\left(\frac{(g-1)^3}{d_1d_2d_3}\right)+h\left(\dfrac{1}{\sqrt{5}}\right)\\
&=\log\left(\max\{(g-1)^3, d_1d_2d_3\} \right)+\dfrac{1}{2}\log 5\\
&=3\log (g-1)+\dfrac{1}{2}\log 5<4\log g
\end{align*}
and
\begin{align*}
h(\alpha)=\dfrac{1}{2}\log\alpha\quad \text{and}\quad h(\eta_3)=\log g.
\end{align*}
Thus, one can take
$$
A_1=8\log g,\quad A_2=\log\alpha\quad \text{and}\quad A_3=2\log g.
$$
Then, we apply Theorem~\ref{Matveev} and find
\begin{align}\label{lambda_1}
\log|\Lambda_1|>-1.4\cdot 30^6\cdot 3^{4.5}\cdot 4\cdot (1+\log 2)\cdot (1+\log (10n\log g))\cdot A
\end{align}
where $A:=A_1A_2A_3=16\log\alpha\cdot\log^2g.$
Comparing inequality \eqref{lambda_1} with \eqref{equation13} gives
\begin{align}
\ell \log g-\log8<7.47\times 10^{12}\cdot (1+\log (10n\log g))\cdot \log^2g.
\end{align}
Because for $g\ge 2$ and $n\ge 2,$
$$
1+\log (10n\log g)<10\log n\cdot \log g,
$$
we would get 
\begin{align}\label{upper of ell}
\ell<7.5\times 10^{13}\cdot \log n\cdot \log^2g.
\end{align}
Secondly, we rewrite \eqref{main_equation1} as
\[
\dfrac{\alpha^k (g-1)}{d_1(g^\ell-1)\sqrt{5}}-\dfrac{\beta^k (g-1)}{d_1(g^\ell-1)\sqrt{5}}=\dfrac{d_2d_3}{(g-1)^2}\left(g^{n+m}-g^m-g^n+1 \right),
\]
which implies
\begin{align}\label{Absolute}
 \dfrac{\alpha^k (g-1)}{d_1(g^\ell-1)\sqrt{5}}-\dfrac{d_2d_3g^{n+m}}{(g-1)^2}&=\dfrac{\beta^k (g-1)}{d_1(g^\ell-1)\sqrt{5}}-\dfrac{d_2d_3g^m}{(g-1)^2}- \dfrac{d_2d_3g^n}{(g-1)^2}\\
\nonumber &+\dfrac{d_2d_3}{(g-1)^2}.
\end{align}
Taking the absolute values of both sides of \eqref{Absolute}, we have
\begin{align*}
 \left| \dfrac{\alpha^k (g-1)}{d_1(g^\ell-1)\sqrt{5}}-\dfrac{d_2d_3g^{n+m}}{(g-1)^2}\right|&\le \dfrac{ (g-1)}{d_1(g^\ell-1)\alpha^k\sqrt{5}}+\dfrac{d_2d_3g^m}{(g-1)^2}+ \dfrac{d_2d_3g^n}{(g-1)^2}\\
 &+\dfrac{d_2d_3}{(g-1)^2}.
\end{align*}
Dividing the both sides of the inequality above by $\dfrac{d_2d_3g^{n+m}}{(g-1)^2}$ and using the fact that $n\ge 2,$ leads to 
\begin{align*}
 \left|\dfrac{(g-1)^3}{d_1d_2d_3(g^\ell-1)\sqrt{5}}\cdot\alpha^k\cdot g^{-(n+m)}-1\right|&\le \dfrac{(g-1)^3}{d_1d_2d_3(g^\ell-1)\alpha^k g^{n+m} \sqrt{5}}\\
 &+\dfrac{1}{g^n}+\dfrac{1}{g^m}+\dfrac{1}{g^{n+m}}<4\cdot g^{-m}.
\end{align*}
Therefore,
\begin{align}\label{Absolute_1}
\left|\dfrac{(g-1)^3}{d_1d_2d_3(g^\ell-1)\sqrt{5}}\cdot\alpha^k\cdot g^{-(n+m)}-1\right|<\dfrac{4}{g^m}.
\end{align}
Now, let us apply Theorem~\ref{Matveev} with
$$
(\eta_1, b_1):=\left(\dfrac{(g-1)^3}{d_1d_2d_3(g^\ell-1)\sqrt{5}}, 1 \right),  (\eta_3, b_3):=(g, -n-m),
$$
$ (\eta_2, b_2):=(\alpha, k),\; s:=3$ and $B:=10 n\log g.$ Note that the numbers $\eta_1, \eta_2$ and $\eta_3$ are
positive real numbers and elements of the ﬁeld $\bL= \bQ(\sqrt{5}).$ It is obvious that
the degree of the ﬁeld $\bL$ is $2.$ So $d_{\bL}=2.$ Let
$$
\Lambda_2:=\dfrac{(g-1)^3}{d_1d_2d_3(g^\ell-1)\sqrt{5}}\cdot\alpha^k\cdot g^{-(n+m)}-1.
$$
If $\Lambda_2=0$, then we get
$$
\alpha^{2k}=\dfrac{5(d_1d_2d_3)^2(g^\ell-1)^2g^{2(n+m)}}{(g-1)^6}\in \bQ(\sqrt{5}).
$$
This is impossible as $\alpha^{2k}$ is irrational for $k\ge 1.$ Therefore, $\Lambda_2$ is nonzero. Moreover, since
\begin{align*}
h(\eta_1&=h\left(\dfrac{(g-1)^3}{d_1d_2d_3(g^\ell-1)\sqrt{5}} \right)\le h\left(\dfrac{(g-1)^3}{d_1d_2d_3} \right) +h\left( \dfrac{1}{g^\ell-1}\right)+h\left(\dfrac{1}{\sqrt{5}}\right)\\
&=\log \max\{(g-1)^3, d_1d_2d_3\}+\log (g^\ell-1)+\dfrac{1}{2}\log 5\\
&<3\log (g-1)+\log (g^\ell-1)+\dfrac{1}{2}\log 5<(5+\ell)\log g,
\end{align*}
and
$$
h(\eta_2)=\dfrac{1}{2}\log \alpha,\quad h(\eta_3)=\log g,
$$
we can take $A_1:=2(\ell+5)\log g,\; A_2:=\log\alpha$ and $A_3:=2\log g.$ Thus, taking into account the inequality \eqref{Absolute_1} and using Theorem~\ref{Matveev}, we obtain
$$
m\log g-\log 4<1.4\cdot 30^6\cdot 3^{4.5}\cdot 4\cdot (1+\log 2)\cdot (1+\log B)\cdot A
$$
where $A=A_1A_2A_3=4(5+\ell)\log\alpha\cdot \log^2g.$  Since $g\ge 2$ and $1+\log B<10\log n\log g,$ it follows that 
\begin{align}\label{upper em1}
m<1.87\times 10^{13}\cdot (5+\ell)\cdot\log n\cdot\log^2 g+2.
\end{align}
By \eqref{upper of ell}, we have
\begin{align}\label{5ell}
5+\ell<7.6\times 10^{13}\cdot \log n\cdot \log^2g.
\end{align}
Therefore from \eqref{upper em1} and \eqref{5ell}, we esealy get 
\begin{align}\label{mm}
m<1.5\times 10^{27}\log^2 n\log^4 g.
\end{align}
Rearranging now equation \eqref{main_equation1} as 
\[
\dfrac{d_3 g^n}{g-1}-\dfrac{(g-1)^2\alpha^k}{d_1d_2(g^\ell-1)(g^m-1)\sqrt{5}}=\dfrac{d_3}{g-1}-\dfrac{\beta^k(g-1)^2}{d_1d_2(g^\ell-1)(g^m-1)\sqrt{5}}
\]
and taking absolute values of both sides of the equality above, we get
\begin{align}\label{aabsolute}
\left| \dfrac{d_3 g^n}{g-1}-\dfrac{(g-1)^2\alpha^k}{d_1d_2(g^\ell-1)(g^m-1)\sqrt{5}}\right|\le  \dfrac{d_3}{g-1}+\dfrac{(g-1)^2}{\alpha^kd_1d_2(g^\ell-1)(g^m-1)\sqrt{5}}.
\end{align}
Dividing both sides of \eqref{aabsolute} by $\dfrac{d_3g^n}{g-1}$ and using the fact that $n\ge 2$, we obtain
\begin{align}\label{aLambd_3}
\left|1- \dfrac{(g-1)^3}{d_1d_2d_3(g^\ell-1)(g^m-1)\sqrt{5}}\cdot g^{-n}\cdot \alpha^k\right|<\dfrac{1}{g^n}+\dfrac{1}{g^{n-1}}<\dfrac{2}{g^{n-1}}.
\end{align}
Put
\begin{align}\label{Lambda_3}
\Lambda_3:=\dfrac{(g-1)^3}{d_1d_2d_3(g^\ell-1)(g^m-1)\sqrt{5}}\cdot g^{-n}\cdot \alpha^k-1.
\end{align}
Next, we apply Theorem~\ref{Matveev} on \eqref{Lambda_3}. First, we need to check that $\Lambda_3\ne 0$. If it were, then we would get that
$$
\alpha^{2k}=\dfrac{5(d_1d_2d_3)^2(g^\ell-1)^2(g^m-1)^2g^{2n}}{(g-1)^6}\in \bQ(\sqrt{5})
$$
which is false. Thus, $\Lambda_3\ne 0.$ So, we apply Theorem~\ref{Matveev} on \eqref{Lambda_3} with the data:
$$
s:=3,\; (\eta_1, b_1):=\left(\dfrac{(g-1)^3}{d_1d_2d_3(g^\ell-1)(g^m-1)\sqrt{5}}, 1 \right),  (\eta_2, b_2):=(g, -n),
$$
and $ (\eta_3, b_3):=(\alpha, k).$ Because, $B\ge \max\{|b_1|,|b_2|,|b_3|\}=\max\{1, n,k\}$ and using the inequality from Lemma~\ref{k} we can take $B:=10 n\log g.$ Note that 
$\eta_1, \eta_2, \eta_3\in \bQ(\alpha)$. Observe that $\bL:=\bQ(\eta_1, \eta_2, \eta_3)=\bQ(\alpha),$ so $d_\bL=2.$ Next,
\begin{align*}
h(\eta_1)&=h\left(\dfrac{(g-1)^3}{d_1d_2d_3(g^\ell-1)(g^m-1)\sqrt{5}} \right)\\
&\le h\left(\dfrac{(g-1)^3}{d_1d_2d_3} \right)+h\left( (g^\ell-1)(g^m-1)\right)+h(\sqrt{5})\\
&< (3+\ell+m)\log g+\dfrac{1}{2}\log 5<(5+\ell+m)\log g.
\end{align*}
Thus, we can take
\[
A_1:=2(5+\ell+m)\log g,\quad A_2:=2\log g\quad \text{and}\quad A_3:=\log\alpha.
\]
Theorem~\ref{Matveev} tells us that
\begin{align}
\log|\Lambda_3|>-1.4\cdot 30^6\cdot 3^{4.5}\cdot 4\cdot (1+\log 2)\cdot (1+\log B)\cdot A
\end{align}
with 
\begin{align}
A=A_1A_2A_3=2\log g\cdot \log\alpha \cdot (10+2\ell+2m)\log g
\end{align}
and 
\begin{align}
1+\log B<10\log n\log g.
\end{align}
Combining the above three inequalities with \eqref{aLambd_3} gives
$$
(n-1)\log g-\log2<1.87\times 10^{13}(5+\ell+m)\log n\log^3 g,
$$
which leads to
\begin{align}\label{nn}
n<1.87\times 10^{13}(5+\ell+m)\log n\log^2 g+2.
\end{align}
Referring to the inqualities \eqref{upper of ell}, \eqref{mm} and \eqref{nn}, we have
\begin{align}
\nonumber 5+\ell+m&<5+7.5\times 10^{13}\cdot \log n\cdot \log^2g+1.5\times 10^{27}\log^2 n\log^4 g\\
&<3\times 10^{27}\log^2 n\log^4 g.
\end{align}
Inserting this in \eqref{nn} leads to 
\begin{align}
n<5.7\times 10^{40}\cdot \log^3 n\cdot\log^6 g.
\end{align}
We are now in position to apply Lemma~\ref{lemma_G-Luca} with the data
$$
l=3,\quad L=n\quad \text{and}\quad H:=5.7\times 10^{40}\cdot\log^6 g.
$$
Therefore, 
\begin{align*}
n&<2^3\cdot 5.7\times 10^{40}\cdot\log^6 g\times \log^3\left(5.7\times 10^{40}\cdot\log^6 g \right)\\
&<2^3\cdot 5.7\times 10^{40}\cdot\log^6 g\cdot (93.84+6\log\log g)^3\\
&<1.08\times 10^{48}\log^9 g.
\end{align*}
In the above inequality, we have used the fact that $93.84+6\log\log g<133\log g,$ which holds for all $g\ge 2.$ Hence, we summarize that all the solutions of \eqref{main_equation1} satisfy
$$
n<1.08\times 10^{48}\log^9 g\quad \text{and}\quad k<10n\log g<1.08\times 10^{49}\log^{10} g.
$$
Hence, the proof of Theorem~\ref{main_result1} is finished.$\quad\quad \quad\quad \quad\quad\; \quad\quad\quad\quad$ $\square$

\begin{remark}
The inequalities from Theorem~\ref{main_result1} allows one to compute all the solutions to \eqref{main_equation1} for every fixed $g.$
\end{remark}

\subsection{Application for the decimal base}

Now, as an illustration, we solve the Diophantine equation \eqref{main_equation1} for $g=10.$ When $g=10,$ the bound
on $k$ becomes
$$
k<4.6\times 10^{52}.
$$
Thus, our main result in this case is the following.
\begin{theorem}\label{Application_Fibo}
The only Fibonacci numbers which are products of three repdigits are
$
1, 2, 3, 5, 8, 21, 55\; \text{and}\; 144.
$

\end{theorem}

\begin{proof}
First, we have to reduce the bounds on $\ell, m, n$ and $k.$ Put
\begin{align*}
z_1&:=\log(\Lambda_1+1)\\
          &=k\log\alpha-(\ell+m+n)\log 10+\log\left(\dfrac{729}{d_1d_2d_3\sqrt{5}}\right).
\end{align*}
Inequality \eqref{equation13} can be written as
\[
\left|e^{z_1}-1\right|<\dfrac{8}{10^\ell}.
\]
Since $\ell\ge 1,$ then we have $\left|e^{z_1}-1\right|<\dfrac{8}{10^\ell}<\dfrac{4}{5}$ which implies that $\dfrac{5}{9}<e^{-z_1}<5.$ If $z_1>0,$ then
$$
0<z_1<e^{z_1}-1=\left|e^{z_1}-1\right|<\dfrac{8}{10^\ell}.
$$
If $z_1<0$, then 
$$
0<|z_1|<e^{|z_1|}-1=e^{-z_1}(1-e^{z_1})<\dfrac{40}{10^\ell}.
$$
In any case, it is always holds true $0<|z_1|<\dfrac{40}{10^\ell},$ which implies
\begin{align}\label{ineq_DP1}
0<\left|k\dfrac{\log \alpha}{\log 10}-(\ell+m+n)+\dfrac{\log\left(729/d_1d_2d_3\sqrt{5}\right)}{\log 10}\right|<17.4\cdot 10^{-\ell}.
\end{align}
It is easy to see that $\dfrac{\log \alpha}{\log 10}$ is irrational. In fact, if $\dfrac{\log \alpha}{\log 10}=\dfrac{p}{q}$ ($p, q\in \bZ$ and $p>0, q>0,$ $\gcd(p, q)=1$), then $10^p=\alpha^q\in \bZ$ which is an absurdity. Now, we apply Lemma~\ref{Dujella-Peto} with $w:=\ell,$
\[
\tau:=\dfrac{\log \alpha}{\log10},\quad \mu:=\dfrac{\log\left(729/d_1d_2d_3\sqrt{5}\right)}{\log 10},\quad A:=17.4,\quad B:=10.
\]
Because $k<4.6\times 10^{52},$ we can take $M:=4.6\times 10^{52}$. Therefore, for the remaining proof, we use Mathematica to apply Lemma~\ref{Dujella-Peto}. For the computations, if the first convergent such that $q > 6M$ does not
satisfy the condition $\varepsilon>0,$ then we use the next convergent until we
find the one that satisfies the conditions. Then we found that the denominator of the $115rd$ convergent
$$
\dfrac{p_{115}}{q_{115}}=\dfrac{1 532 282 514 732 971 248 699 360 262 855 137 347 685 624 203 086 792 614}{7 331 928 878 186 982 501 184 370 491 249 297 952 824 659 131 062 806 099}
$$
of $\tau$ exceeds $6M.$ Thus, we can say that the inequality \eqref{ineq_DP1} has no solution for
$$
\ell=w\ge \dfrac{\log (Aq_{115}/\varepsilon)}{\log 10}\ge\dfrac{\log(Aq_{115}/0.00809526)}{\log 10}\ge 58.1975.
$$
So, we obtain
\begin{align}\label{ell_1}
\ell\le 58.
\end{align}
Substituting this upper bound for $\ell$ into \eqref{upper em1} and combining the new bound obtained with \eqref{nn}, we get 
$$
 n<6.4\times10^{29}\cdot \log^2 n,
$$
which implies $n<1.3\times 10^{34}$ using Lemma~\ref{lemma_G-Luca}. Thus, by Lemma~\ref{k} we have $k<3\times 10^{35}.$ Next, we need to reduce the bound on $m.$ We return to \eqref{Absolute_1} and put
\begin{align*}
z_2&:=\log(\Lambda_2+1)\\
         &=k\log\alpha-(n+m)\log 10+\log\left(\dfrac{729}{d_1d_2d_3(10^\ell-1)\sqrt{5}} \right).
\end{align*}
From the inequality \eqref{Absolute_1} and $m\ge 1$, we conclude that
\[
\left|e^{z_2}-1\right|<\dfrac{4}{10^m}<\dfrac{1}{2},
\]
which implies that
$\dfrac{1}{2}<e^{z_2}<\dfrac{3}{2}.$ If $z_2>0,$ then $0<z_2<e^{z_2}-1<\dfrac{4}{10^m}.$ If $z_2<0$, then
$$
 0<|z_2|<e^{|z_2|}-1=e^{-z_2}-1=e^{-z_2}(1-e^{z_2})<\dfrac{8}{10^m}.
$$ 
In any case, since $0<|z_2|<\dfrac{8}{10^m},$ thus we have
\begin{align}\label{ineq_DP2}
0<\left| k\dfrac{\log \alpha}{\log 10}-(n+m)+\dfrac{\log\left(729/d_1d_2d_3(10^\ell-1)\sqrt{5} \right)}{\log 10}\right|<\dfrac{3.5}{10^m}.
\end{align}
Again, we apply Lemma~\ref{Dujella-Peto} with
\[
\tau:=\dfrac{\log \alpha}{\log 10},\; \mu:=\dfrac{\log\left(729/d_1d_2d_3(10^\ell-1)\sqrt{5} \right)}{\log 10},\; A:=3.5,\; B:=10
\]
and 
$M:=3\times 10^{35}.$ With the help of Mathematica, we found that the denominator of the $77rd$ convergent
$$
\dfrac{p_{77}}{q_{77}}=\dfrac{1 097 876 139 463 713 781 430 275 039 172 749 779}{5 253 306 550 332 349 137 376 600 680 873 772 748}
$$
of $\tau$ exceeds $6M.$ It follows that the inequality \eqref{ineq_DP2} has no solution for
$$
m=w\ge \dfrac{\log (Aq_{77}/\varepsilon)}{\log 10}\ge\dfrac{\log(Aq_{77}/0.000111931)}{\log 10}\ge 41.2155.
$$
Hence, we obtain
\begin{align}\label{m_1}
m\le 41.
\end{align}
Inserting the bounds from \eqref{ell_1} and \eqref{m_1} in \eqref{nn}, we get 
$$
n<1.1\times 10^{16}\log n+2,
$$
which leads to $n<8.2\times 10^{17}$ and hence to $k<2\times 10^{19}.$ Finaly, we have to reduce the bound on $n.$ From \eqref{aLambd_3}, we can put
\begin{align*}
z_3&:=\log(\Lambda_3+1)\\
         &=k\log\alpha-n\log 10+\log\left(\dfrac{729}{d_1d_2d_3(10^\ell-1)(10^m-1)\sqrt{5}} \right).
\end{align*}
By following what is done in previous cases, it is easy to see that for $n\ge 2,$ we have
\begin{align}\label{ineq_DP3}
0<\left| k\dfrac{\log \alpha}{\log 10}-n+\dfrac{\log\left(729/d_1d_2d_3(10^\ell-1)(10^m-1)\sqrt{5} \right)}{\log 10}\right|<\dfrac{1.8}{10^{n-1}}.
\end{align}
Now, we apply Lemma~\ref{Dujella-Peto} to \eqref{ineq_DP3} with $B:=10,$
\[
\tau:=\dfrac{\log \alpha}{\log 10},\; \mu:=\dfrac{\log\left(729/d_1d_2d_3(10^\ell-1)(10^m-1)\sqrt{5} \right)}{\log 10},\; A:=1.8,
\]
and 
$M:=2\times 10^{19}.$ We saw that the denominator of the $44rd$ convergent
$$
\dfrac{p_{44}}{q_{44}}=\dfrac{259 791 952 914 951 895 804}{1 243 097 211 893 507 332 887}
$$
of $\tau$ exceeds $6M.$ Therefore the inequality \eqref{ineq_DP3} has no solution for
$$
n-1=w\ge \dfrac{\log (Aq_{44}/\varepsilon)}{\log 10}\ge\dfrac{\log(Aq_{44}/0.0000637147)}{\log 10}\ge 25.5455.
$$
Hence, we obtain
\begin{align}
n\le 26.
\end{align}
So, it remains to check equation \eqref{main_equation1} in the case $g=10$ for $1\le d_1,d_2,d_3\le 9$, $1 \le n \le 26,$ $1 \le k \le 598,$ $1\le \ell\le 58$ and $1\le m \le 41.$ A quick inspection using Maple reveals that the Diophantine equation \eqref{main_equation1} has only the solutions listed in the statement of Theorem~\ref{Application_Fibo}. This ends the proof of Theorem~\ref{Application_Fibo}.
\end{proof}

\section{Lucas numbers as products of three repdgits in base $g$}\label{Sec_Lucas}

In this section, we will follow the method from Section~\ref{Sec_Fibo}. For the sake of completeness, we will give most of details. Our first aim is to prove Theorem~\ref{main_result2}.

\subsection{Proof of Theorem~\ref{main_result2}}

Here too by taking $n=1,$ we easily verify that the bound of $k$ from Theorem~\ref{main_result2} is valid. Now let us see what happens for $n\ge 2$. The following result will be useful in proving our main result which gives a relation between $n$ and $k$ of equation \eqref{main_equation2}.
\begin{lemma}\label{k1}
All solutions of the Diophantine equation \eqref{main_equation2} satisfy
$$
k<3n\dfrac{\log g}{\log \alpha}+1<10n\log g.
$$
\end{lemma}
\begin{proof}
From \eqref{F_n-L_n}, we have
$$  
\alpha^{k-1}\leq L_k=d_1\frac{g^\ell-1}{g-1}\cdot d_2\frac{g^m-1}{g-1}\cdot d_3\frac{g^n-1}{g-1}\le  (g^n-1)^3<g^{3n}.
$$
Taking logarithm on both sides, we get $(k-1)\log\alpha<3n\log g.$ Since, $n\ge 2$ and $g\ge 2,$ we obtain the desired inequalities. This ends the proof. 
\end{proof}

First, we find the upper bounds for the variables $n, \ell, m$ of equation \eqref{main_equation2}. From \eqref{FLn} and \eqref{main_equation2}, we get 
$$  
L_k=\alpha^k+\beta^k=d_1\frac{g^\ell-1}{g-1}\cdot d_2\frac{g^m-1}{g-1}\cdot d_3\frac{g^n-1}{g-1}
$$
to obtain 
\begin{align}\label{1equation11}
\alpha^k-\frac{d_1d_2d_3g^{\ell+m+n}}{(g-1)^3}&=-\beta^k-\frac{d_1d_2d_3g^{\ell+m}}{(g-1)^3}-\frac{d_1d_2d_3g^{n+\ell}}{(g-1)^3}+\frac{d_1d_2d_3g^\ell}{(g-1)^3}\\
&\nonumber -\frac{d_1d_2d_3g^{n+m}}{(g-1)^3}+\frac{d_1d_2d_3g^m}{(g-1)^3}+\frac{d_1d_2d_3g^n}{(g-1)^3}-\frac{d_1d_2d_3}{(g-1)^3}.
\end{align}
Taking the absolute values of both sides of \eqref{1equation11}, we get 
\begin{align}\label{1equation12}
\left|\alpha^k-\frac{d_1d_2d_3g^{\ell+m+n}}{(g-1)^3}\right|&\le\dfrac{1}{\alpha^k}+\frac{d_1d_2d_3g^{\ell+m}}{(g-1)^3}+\frac{d_1d_2d_3g^{n+\ell}}{(g-1)^3}+\frac{d_1d_2d_3g^\ell}{(g-1)^3}\\
&\nonumber +\frac{d_1d_2d_3g^{n+m}}{(g-1)^3}+\frac{d_1d_2d_3g^m}{(g-1)^3}+\frac{d_1d_2d_3g^n}{(g-1)^3}+\frac{d_1d_2d_3}{(g-1)^3}.
\end{align}
Dividing both sides of \eqref{1equation12} by $\frac{d_1d_2d_3g^{\ell+n+m}}{(g-1)^3}$ and using the fact that $\ell\le m\le n$ gives us the following inequalities.
\begin{align}
\nonumber \left|\frac{(g-1)^3\cdot\alpha^k\cdot g^{-(\ell+n+m)}}{d_1d_2d_3}-1 \right|& \leq \frac{(g-1)^3}{\alpha^k d_1d_2d_3g^{\ell+n+m}}+\frac{1}{g^n}+\frac{1}{g^m}+\frac{1}{g^{n+m}}\\
\notag &+\frac{1}{g^\ell}+\frac{1}{g^{n+\ell}}+\frac{1}{g^{\ell+m}}+\dfrac{1}{g^{\ell+m+n}}\\
\notag &< 8\cdot g^{-\ell}.
\end{align}
From this, it follows that 
\begin{equation}\label{1equation13}
\left|\frac{(g-1)^3}{d_1d_2d_3}\cdot\alpha^k\cdot g^{-(\ell+n+m)}-1 \right| < \frac{8}{g^\ell}.
\end{equation}
Let
$$
\Lambda_4:=\frac{(g-1)^3}{d_1d_2d_3}\cdot\alpha^k\cdot g^{-(\ell+n+m)}-1.
$$
We need to show $\Lambda_4\ne 0$. Suppose $\Lambda_4=0$, then
$$
\alpha^{n}=\dfrac{d_1d_2d_3}{(g-1)^3}\cdot g^{\ell+m+n}\in \Q(\sqrt{5})
$$
which is false. Thus $\Lambda_4\ne 0.$ To apply Theorem~\ref{Matveev} to $\Lambda_4,$ we choose the following data
$$
(\eta_1, b_1):=\left(\frac{(g-1)^3}{d_1d_2d_3},1 \right),\; (\eta_2, b_2):=(\alpha, k),\; (\eta_3, b_3):=(g, -(\ell+m+n))
$$
and $s:=3.$ Note that $\eta_1, \eta_2, \eta_3\in \bQ(\alpha)$ and $b_1, b_2, b_3\in \bZ.$ The degree $d_\bL=[\bQ(\alpha):\bQ]$ is $2$, where $\bL$ is $\bQ(\alpha).$ According to the inequalities from Lemma~\ref{k1}, we can take $B:=10n\log g.$ To estimate the parameters $A_1, A_2, A_3 ,$ we calculate the logarithmic heights
of $\eta_1, \eta_2, \eta_3$ as follows:
\begin{align*}
h(\eta_1)=h\left(\frac{(g-1)^3}{d_1d_2d_3} \right)
&=\log\left(\max\{(g-1)^3, d_1d_2d_3\} \right)\\
&=3\log (g-1)<3\log g
\end{align*}
and
\begin{align*}
h(\alpha)=\dfrac{1}{2}\log\alpha\quad \text{and}\quad h(\eta_3)=\log g.
\end{align*}
Thus, one can take
$$
A_1=6\log g,\quad A_2=\log\alpha\quad \text{and}\quad A_3=2\log g.
$$
Then, we apply Theorem~\ref{Matveev} and find
\begin{align}\label{1lambda_1}
\log|\Lambda_4|>-1.4\cdot 30^6\cdot 3^{4.5}\cdot 4\cdot (1+\log 2)\cdot (1+\log (10n\log g))\cdot A
\end{align}
where $A:=A_1A_2A_3=12\log\alpha\cdot\log^2g.$
Comparing inequality \eqref{1lambda_1} with \eqref{1equation13} gives
\begin{align}
\ell \log g-\log8<5.6\times 10^{12}\cdot (1+\log (10n\log g))\cdot \log^2g.
\end{align}
Note that for $g\ge 2$ and $n\ge 2,$
$$
1+\log (10n\log g)<10\log n\cdot \log g,
$$
we would get 
\begin{align}\label{1upper of ell}
\ell<5.7\times 10^{13}\cdot \log n\cdot \log^2g.
\end{align}

Now, we rewrite \eqref{main_equation2} as
\[
\dfrac{\alpha^k (g-1)}{d_1(g^\ell-1)}+\dfrac{\beta^k (g-1)}{d_1(g^\ell-1)}=\dfrac{d_2d_3}{(g-1)^2}\left(g^{n+m}-g^m-g^n+1 \right),
\]
which implies
\begin{align}\label{1Absolute}
 \dfrac{\alpha^k (g-1)}{d_1(g^\ell-1)}-\dfrac{d_2d_3g^{n+m}}{(g-1)^2}&=-\dfrac{\beta^k (g-1)}{d_1(g^\ell-1)}-\dfrac{d_2d_3g^m}{(g-1)^2}- \dfrac{d_2d_3g^n}{(g-1)^2}\\
\nonumber &+\dfrac{d_2d_3}{(g-1)^2}.
\end{align}
Taking the absolute values of both sides of \eqref{1Absolute}, we have
\begin{align*}
 \left| \dfrac{\alpha^k (g-1)}{d_1(g^\ell-1)}-\dfrac{d_2d_3g^{n+m}}{(g-1)^2}\right|&\le \dfrac{ (g-1)}{d_1(g^\ell-1)\alpha^k}+\dfrac{d_2d_3g^m}{(g-1)^2}+ \dfrac{d_2d_3g^n}{(g-1)^2}\\
 &+\dfrac{d_2d_3}{(g-1)^2}.
\end{align*}
Dividing the both sides of the inequality above by $\dfrac{d_2d_3g^{n+m}}{(g-1)^2}$ and using the fact that $n\ge 2,$ leads to 
\begin{align*}
 \left|\dfrac{(g-1)^3}{d_1d_2d_3(g^\ell-1)}\cdot\alpha^k\cdot g^{-(n+m)}-1\right|&\le \dfrac{(g-1)^3}{d_1d_2d_3(g^\ell-1)\alpha^k g^{n+m} }\\
 &+\dfrac{1}{g^n}+\dfrac{1}{g^m}+\dfrac{1}{g^{n+m}}<4\cdot g^{-m}.
\end{align*}
Therefore,
\begin{align}\label{1Absolute_1}
\left|\dfrac{(g-1)^3}{d_1d_2d_3(g^\ell-1)}\cdot\alpha^k\cdot g^{-(n+m)}-1\right|<\dfrac{4}{g^m}.
\end{align}
Now, let us apply Theorem~\ref{Matveev} with
$$
(\eta_1, b_1):=\left(\dfrac{(g-1)^3}{d_1d_2d_3(g^\ell-1)}, 1 \right),  (\eta_3, b_3):=(g, -n-m),
$$
$ (\eta_2, b_2):=(\alpha, k),\; s:=3$ and $B:=10 n\log g.$ Note that the numbers $\eta_1, \eta_2$ and $\eta_3$ are
positive real numbers and elements of the ﬁeld $\bL= \bQ(\alpha).$ It is obvious that
the degree of the ﬁeld $\bL$ is $2.$ So $d_{\bL}=2.$ Let
$$
\Lambda_5:=\dfrac{(g-1)^3}{d_1d_2d_3(g^\ell-1)}\cdot\alpha^k\cdot g^{-(n+m)}-1.
$$
If $\Lambda_5=0$, then we get
$$
\alpha^{k}=\dfrac{d_1d_2d_3(g^\ell-1)g^{n+m}}{(g-1)^3}\in \bQ(\sqrt{5}).
$$
This is impossible as $\alpha^k$ is irrational for $k\ge 1.$ Therefore, $\Lambda_5$ is nonzero. Moreover, since
\begin{align*}
h(\eta_1&=h\left(\dfrac{(g-1)^3}{d_1d_2d_3(g^\ell-1)} \right)\le h\left(\dfrac{(g-1)^3}{d_1d_2d_3} \right) +h\left( \dfrac{1}{g^\ell-1}\right)\\
&=\log \max\{(g-1)^3, d_1d_2d_3\}+\log (g^\ell-1)\\
&<3\log (g-1)+\log (g^\ell-1)<(3+\ell)\log g,
\end{align*}
and
$$
h(\eta_2)=\dfrac{1}{2}\log \alpha,\quad h(\eta_3)=\log g,
$$
we can take $A_1:=2(\ell+3)\log g,\; A_2:=\log\alpha$ and $A_3:=2\log g.$ Thus, taking into account the inequality \eqref{1Absolute_1} and using Theorem~\ref{Matveev}, we obtain
$$
m\log g-\log 4<1.4\cdot 30^6\cdot 3^{4.5}\cdot 4\cdot (1+\log 2)\cdot (1+\log B)\cdot A
$$
where $A=A_1A_2A_3=4(3+\ell)\log\alpha\cdot \log^2g.$  Since $g\ge 2$ and $1+\log B<10\log n\log g,$
we obtain that
\begin{align}\label{1upper em1}
m<1.87\times 10^{13}\cdot (3+\ell)\cdot\log n\cdot\log^2 g+2.
\end{align}
By \eqref{1upper of ell}, we have
\begin{align}\label{5ell1}
3+\ell<5.8\times 10^{13}\cdot \log n\cdot \log^2g.
\end{align}
Therefore from \eqref{1upper em1} and \eqref{5ell1}, we esealy get 
\begin{align}\label{1mm}
m<1.1\times 10^{27}\log^2 n\log^4 g.
\end{align}

Rearranging now equation \eqref{main_equation2} as 
\[
\dfrac{d_3 g^n}{g-1}-\dfrac{(g-1)^2\alpha^k}{d_1d_2(g^\ell-1)(g^m-1)}=\dfrac{d_3}{g-1}+\dfrac{\beta^k(g-1)^2}{d_1d_2(g^\ell-1)(g^m-1)}
\]
and taking absolute values of both sides of the equality above, we get
\begin{align}\label{1aabsolute}
\left| \dfrac{d_3 g^n}{g-1}-\dfrac{(g-1)^2\alpha^k}{d_1d_2(g^\ell-1)(g^m-1)}\right|\le  \dfrac{d_3}{g-1}+\dfrac{(g-1)^2}{\alpha^kd_1d_2(g^\ell-1)(g^m-1)}.
\end{align}
Dividing both sides of \eqref{1aabsolute} by $\dfrac{d_3g^n}{g-1}$ and using the fact that $n\ge 2$, we obtain
\begin{align}\label{1aLambd_3}
\left|1- \dfrac{(g-1)^3}{d_1d_2d_3(g^\ell-1)(g^m-1)}\cdot g^{-n}\cdot \alpha^k\right|<\dfrac{1}{g^n}+\dfrac{1}{g^{n-1}}<\dfrac{2}{g^{n-1}}.
\end{align}
Put
\begin{align}\label{1Lambda_3}
\Lambda_6:=\dfrac{(g-1)^3}{d_1d_2d_3(g^\ell-1)(g^m-1)}\cdot g^{-n}\cdot \alpha^k-1.
\end{align}
Next, we apply Theorem~\ref{Matveev} on \eqref{1Lambda_3}. First, we need to check that $\Lambda_6\ne 0$. If it were, then we would get that
$$
\alpha^{k}=\dfrac{d_1d_2d_3(g^\ell-1)(g^m-1)g^n}{(g-1)^3}\in \bQ(\sqrt{5})
$$
which is false. Thus, $\Lambda_6\ne 0.$ So, we apply Theorem~\ref{Matveev} on \eqref{1Lambda_3} with the data:
$$
s:=3,\; (\eta_1, b_1):=\left(\dfrac{(g-1)^3}{d_1d_2d_3(g^\ell-1)(g^m-1)}, 1 \right),  (\eta_2, b_2):=(g, -n),
$$
and $ (\eta_3, b_3):=(\alpha, k).$ Because, $B\ge \max\{|b_1|,|b_2|,|b_3|\}=\max\{1, n,k\}$ and using the inequality from Lemma~\ref{k1} we can take $B:=10 n\log g.$ Note that 
$\eta_1, \eta_2, \eta_3\in \bQ(\alpha)$. Observe that $\bL:=\bQ(\eta_1, \eta_2, \eta_3)=\bQ(\alpha),$ so $d_\bL=2.$ Also,
\begin{align*}
h(\eta_1)&=h\left(\dfrac{(g-1)^3}{d_1d_2d_3(g^\ell-1)(g^m-1)} \right)\\
&\le h\left(\dfrac{(g-1)^3}{d_1d_2d_3} \right)+h\left( (g^\ell-1)(g^m-1)\right)
< (3+\ell+m)\log g.
\end{align*}
Thus, we can take
\[
A_1:=2(3+\ell+m)\log g,\quad A_2:=2\log g\quad \text{and}\quad A_3:=\log\alpha.
\]
Theorem~\ref{Matveev} tells us that
\begin{align}
\log|\Lambda_6|>-1.4\cdot 30^6\cdot 3^{4.5}\cdot 4\cdot (1+\log 2)\cdot (1+\log B)\cdot A
\end{align}
with 
\begin{align}
A=A_1A_2A_3=2\log g\cdot \log\alpha \cdot 2(3+\ell+m)\log g
\end{align}
and 
\begin{align}
1+\log B<10\log n\log g.
\end{align}
Combining the above three inequalities with \eqref{1aLambd_3} gives
$$
(n-1)\log g-\log2<1.87\times 10^{13}(3+\ell+m)\log n\log^3 g,
$$
which leads to
\begin{align}\label{1nn}
n<1.87\times 10^{13}(3+\ell+m)\log n\log^2 g+2.
\end{align}
Referring to the inqualities \eqref{1upper of ell}, \eqref{1mm} and \eqref{1nn}, we have
\begin{align}
\nonumber 3+\ell+m&<5.8\times 10^{13}\cdot \log n\cdot \log^2g+1.1\times 10^{27}\log^2 n\log^4 g\\
&<2.2\times 10^{27}\log^2 n\log^4 g.
\end{align}
Inserting this in \eqref{1nn} leads to 
\begin{align}
n<4.2\times 10^{40}\cdot \log^3 n\cdot\log^6 g.
\end{align}
We have everything ready to apply Lemma~\ref{lemma_G-Luca} with the data
$$
l=3,\quad L=n\quad \text{and}\quad H:=4.2\times 10^{40}\cdot\log^6 g.
$$
Therefore, 
\begin{align*}
n&<2^3\cdot 5.7\times 10^{40}\cdot\log^6 g\times \log^3\left(5.7\times 10^{40}\cdot\log^6 g \right)\\
&<2^3\cdot 4.2\times 10^{40}\cdot\log^6 g\cdot (93.6+6\log\log g)^3\\
&<7.73\times 10^{47}\log^9 g.
\end{align*}
In the above inequality, we have used the fact that $93.6+6\log\log g<132\log g,$ which holds for all $g\ge 2.$ Hence, we summarize that all the solutions of \eqref{main_equation2} satisfy
$$
n<7.73\times 10^{47}\log^9 g\quad \text{and}\quad k<10n\log g<7.73\times 10^{48}\log^{10} g.
$$
Hence, the proof of Theorem~\ref{main_result2} is finished.$\quad\quad \quad\quad \quad\quad\; \quad\quad\quad\quad$ $\square$

\begin{remark}
The inequalities from Theorem~\ref{main_result2} allows one to compute all the solutions to \eqref{main_equation2} for every fixed $g.$
\end{remark}

\subsection{Application for the decimal base}

Now, as an illustration, we solve equation \eqref{main_equation2} for $g=10.$ In this case according to Theorem~\ref{main_result2}, the bound on $k$ becomes
$$
k<3.3\times 10^{52}.
$$
Here is our main result in this case.
\begin{theorem}\label{Application_Lucas}
The only Lucas numbers which are products of three repdigits are $1, 3, 4, 7, 11\; \text{and}\; 18.$
\end{theorem}

\begin{proof}
We must first reduce the bounds on $\ell, m, n$ and $k.$ Put
\begin{align*}
z_4&:=\log(\Lambda_4+1)\\
          &=k\log\alpha-(\ell+m+n)\log 10+\log\left(\dfrac{729}{d_1d_2d_3}\right).
\end{align*}
Inequality \eqref{1equation13} can be written as
\[
\left|e^{z_4}-1\right|<\dfrac{8}{10^\ell}.
\]
Since $\ell\ge 1,$ then we have $\left|e^{z_4}-1\right|<\dfrac{8}{10^\ell}<\dfrac{4}{5}$ which implies that $\dfrac{5}{9}<e^{-z_4}<5.$ If $z_4>0,$ then
$$
0<z_4<e^{z_4}-1=\left|e^{z_4}-1\right|<\dfrac{8}{10^\ell}.
$$
If $z_4<0$, then 
$$
0<|z_4|<e^{|z_4|}-1=e^{-z_4}(1-e^{z_4})<\dfrac{40}{10^\ell}.
$$
In any case, it is always holds true $0<|z_4|<\dfrac{40}{10^\ell},$ which implies
\begin{align}\label{1ineq_DP1}
0<\left|k\dfrac{\log \alpha}{\log 10}-(\ell+m+n)+\dfrac{\log\left(729/d_1d_2d_3\right)}{\log 10}\right|<17.4\cdot 10^{-\ell}.
\end{align}
It is easy to see that $\dfrac{\log \alpha}{\log 10}$ is irrational.  Now, we have to study the following two cases.

\textbf{Case~1: $(d_1, d_2, d_3)\ne (9, 9, 9).$} We apply Lemma~\ref{Dujella-Peto} with $w:=\ell,$
\[
\tau:=\dfrac{\log \alpha}{\log10},\quad \mu:=\dfrac{\log\left(729/d_1d_2d_3\right)}{\log 10},\quad A:=17.4,\quad B:=10.
\]
Because $k<3.3\times 10^{52},$ we can take $M:=3.3\times 10^{52}$. We use Mathematica to apply Lemma~\ref{Dujella-Peto} and found that the denominator of the $114rd$ convergent
$$
\dfrac{p_{114}}{q_{114}}=\dfrac{75 199 708 224 715 672 236 920 162 770 429 633 212 096 962 359 234 385}{359 828 495 765 425 172 949 832 316 042 466 402 419 242 364 862 312 251}
$$
of $\tau$ exceeds $6M.$ Thus, we can say that the inequality \eqref{1ineq_DP1} has no solution for
$$
\ell=w\ge \dfrac{\log (Aq_{114}/\varepsilon)}{\log 10}\ge\dfrac{\log(Aq_{114}/0.0028639)}{\log 10}\ge 57.3397.
$$
So, we obtain
\begin{align}\label{11ell_1}
\ell\le 57.
\end{align}

\textbf{Case~2: $(d_1, d_2, d_3)= (9, 9, 9).$}
In this case from \eqref{1ineq_DP1}, we have 
$$
0<\left|k\dfrac{\log \alpha}{\log 10}-(\ell+m+n)\right|<17.4\cdot 10^{-\ell}.
$$
If we divide this inequality by $k$, we get
\begin{align}
0<\left|\dfrac{\log \alpha}{\log 10}-\dfrac{\ell+m+n}{k}\right|<\dfrac{17.4}{k\cdot 10^\ell}. \end{align}
Assume that $\ell >55$. Then it can be seen that
$$
\dfrac{10^\ell}{2(17.4)}>2.87\times 10^{53}>3.3\times 10^{52}>k,
$$
and so we have
$$
\left|\dfrac{\log \alpha}{\log 10}-\dfrac{\ell+m+n}{k}\right|<\dfrac{17.4}{k\cdot 10^\ell}<\dfrac{1}{2k^2}.
$$
From the known properties of continued fraction (Lemma~\ref{Lemma-Legendre}), it is seen that the rational $\dfrac{\ell+m+n}{k}$ is a convergent to $\kappa:=\dfrac{\log \alpha}{\log 10}.$ So $\dfrac{\ell+m+n}{k}$ is of the form $p_t /q_t$ for some $t.$ Then we have
$$
q_{109}>3.3\times 10^{52}>k.
$$
Thus $t\in \{0,1, 2,\ldots,108\}.$  By Lemma~\ref{Lemma-Legendre}, we get
$$
\dfrac{1}{(a(M)+2)\cdot k^2}\le \left|\dfrac{\log \alpha}{\log 10}-\dfrac{\ell+m+n}{k}\right|<\dfrac{17.4}{k\cdot 10^\ell}.
$$
Since, $a(M)= \max\{a_i: i=0, 1, 2,\ldots,108\}=106,$ we get
\begin{align*}
\ell<\dfrac{\log\left(17.4\cdot (106+2)\cdot 3.3\times 10^{52} \right)}{\log 10}<55.8,
\end{align*}
 which contradict the fact that $\ell>55.$ Thus $\ell \le 55.$ Therefore the bound
 \begin{align}\label{1ell_Lu}
 \ell\le 57,
 \end{align}
 holds in all two cases. Substituting this upper bound for $\ell$ into \eqref{1upper em1} and combining the new bound obtained with \eqref{1nn}, we get 
$$
 n<6\times10^{29}\cdot \log^2 n,
$$
which implies $n<1.2\times 10^{34}$ using Lemma~\ref{lemma_G-Luca}. Thus, by Lemma~\ref{k1} we have $k<2.8\times 10^{35}.$ Next, we need to reduce the bound on $m.$ We return to \eqref{1Absolute_1} and put
\begin{align*}
z_5&:=\log(\Lambda_5+1)\\
         &=k\log\alpha-(n+m)\log 10+\log\left(\dfrac{729}{d_1d_2d_3(10^\ell-1)} \right).
\end{align*}
From the inequality \eqref{1Absolute_1} and $m\ge 1$, we conclude that
\[
\left|e^{z_5}-1\right|<\dfrac{4}{10^m}<\dfrac{1}{2},
\]
which implies that
$\dfrac{1}{2}<e^{z_5}<\dfrac{3}{2}.$ If $z_5>0,$ then $0<z_5<e^{z_5}-1<\dfrac{4}{10^m}.$ If $z_5<0$, then
$$
 0<|z_5|<e^{|z_5|}-1=e^{-z_5}-1=e^{-z_5}(1-e^{z_5})<\dfrac{8}{10^m}.
$$ 
In any case, since $0<|z_5|<\dfrac{8}{10^m},$ thus we have
\begin{align}\label{1ineq_DP2}
0<\left| k\dfrac{\log \alpha}{\log 10}-(n+m)+\dfrac{\log\left(729/d_1d_2d_3(10^\ell-1)\right)}{\log 10}\right|<\dfrac{3.5}{10^m}.
\end{align}
Here we have to study the following two cases.

\textbf{Case a: $(\ell, d_1, d_2, d_3)\ne (1, 1,9,9), (1,9,9,1), (1, 9,1,9).$}

Thus, we can apply Lemma~\ref{Dujella-Peto} with the following data
\[
\tau:=\dfrac{\log \alpha}{\log 10},\; \mu:=\dfrac{\log\left(729/d_1d_2d_3(10^\ell-1) \right)}{\log 10},\; A:=3.5,\; B:=10
\]
and 
$M:=2.8\times 10^{35}.$ Using Mathematica, we found that the denominator of the $97rd$ convergent
$$
\dfrac{p_{97}}{q_{97}}=\dfrac{3 106 590 240 739 929 077 205 211 403 373 423 170 367 494 081}{14 864 947 214 218 067 609 395 035 403 916 715 939 116 150 260}
$$
of $\tau$ exceeds $6M.$ It follows that the inequality \eqref{1ineq_DP2} has no solution for
$$
m=w\ge \dfrac{\log (Aq_{97}/\varepsilon)}{\log 10}\ge\dfrac{\log(Aq_{97}/4.58528 \times 10^{-12})}{\log 10}\ge 58.0549.
$$
Hence, we obtain
$
m\le 58.
$
 
 \textbf{Case b: $(\ell, d_1, d_2, d_3)\in \{ (1, 1,9,9), (1,9,9,1), (1, 9,1,9)\}.$}

 In this case from \eqref{1ineq_DP2}, we have 
$$
0<\left|k\dfrac{\log \alpha}{\log 10}-(m+n)\right|<3.5\cdot 10^{-m}.
$$
If we divide this inequality by $k$, we get
\begin{align}
0<\left|\dfrac{\log \alpha}{\log 10}-\dfrac{m+n}{k}\right|<\dfrac{3.5}{k\cdot 10^m}. \end{align}
Assume that $m \ge39$. Then it can be seen that
$$
\dfrac{10^m}{2(3.5)}>10n\log 10>2.8\times 10^{35}>k,
$$
and so we have
$$
\left|\dfrac{\log \alpha}{\log 10}-\dfrac{m+n}{k}\right|<\dfrac{3.5}{k\cdot 10^m}<\dfrac{1}{2k^2}.
$$
 From the known properties of continued fraction (Lemma~\ref{Lemma-Legendre}), it is seen that the rational $\dfrac{m+n}{k}$ is a convergent to $\kappa:=\dfrac{\log \alpha}{\log 10}.$ So $\dfrac{m+n}{k}$ is of the form $p_t /q_t$ for some $t.$ Then we have
$$
q_{73}>2.8\times 10^{35}>10n\log 10>k.
$$
Thus $t\in \{0,1, 2,\ldots,72\}.$  By Lemma~\ref{Lemma-Legendre}, we get
$$
\dfrac{1}{(a(M)+2)\cdot k^2}\le \left|\dfrac{\log \alpha}{\log 10}-\dfrac{m+n}{k}\right|<\dfrac{3.5}{k\cdot 10^m}.
$$
Since, $a(M)= \max\{a_i: i=0, 1, 2,\ldots,72\}=106,$ we get
$$
\dfrac{3.5}{10^{39}}\ge \dfrac{3.5}{10^m}>\dfrac{1}{108\cdot k}>\dfrac{1}{108\cdot 2.8\times 10^{35}},
$$
which is a contradiction. Therefore $m\le 38.$ In all cases $a$ and $b$ we can consider  
\begin{align}\label{1m_1}
m\le 58.
\end{align}

Using the inequalities \eqref{1ell_Lu} and \eqref{1m_1} together and substituting these upper bounds into \eqref{1nn}, we get 
$$
n<1.2\times 10^{16}\log n,
$$
which leads to $n<9\times 10^{17}$ and hence to $k<2.1\times 10^{19}.$ To reduce the bound on $n$ we must return to inequality \eqref{1aLambd_3} and put
\begin{align*}
z_6&:=\log(\Lambda_6+1)\\
         &=k\log\alpha-n\log 10+\log\left(\dfrac{729}{d_1d_2d_3(10^\ell-1)(10^m-1)} \right).
\end{align*}
For $n\ge 2,$ we can easily see that
\begin{align}\label{1ineq_DP3}
0<\left| k\dfrac{\log \alpha}{\log 10}-n+\dfrac{\log\left(729/d_1d_2d_3(10^\ell-1)(10^m-1) \right)}{\log 10}\right|<\dfrac{1.8}{10^{n-1}}.
\end{align}
It is also appropriate here to take into account the following two cases according to the values of the variables.

\textbf{Case I: $729\ne d_1d_2d_3(10^\ell-1)(10^m-1).$}

Therefore, we can apply Lemma~\ref{Dujella-Peto} to \eqref{1ineq_DP3} with $B:=10,$
\[
\tau:=\dfrac{\log \alpha}{\log 10},\; \mu:=\dfrac{\log\left(729/d_1d_2d_3(10^\ell-1)(10^m-1) \right)}{\log 10},\; A:=1.8,
\]
and 
$M:=2.1\times 10^{19}.$ We saw that the denominator of the $44rd$ convergent
$$
\dfrac{p_{44}}{q_{44}}=\dfrac{259 791 952 914 951 895 804}{1 243 097 211 893 507 332 887}
$$
of $\tau$ exceeds $6M.$ Therefore the inequality \eqref{1ineq_DP3} has no solution for
$$
n-1=w\ge \dfrac{\log (Aq_{44}/\varepsilon)}{\log 10}\ge\dfrac{\log(Aq_{44}/0.0000149094)}{\log 10}\ge 26.1763.
$$
Hence, we obtain
\begin{align}
n\le 27.
\end{align}

\textbf{Case II: $729= d_1d_2d_3(10^\ell-1)(10^m-1).$}

From \eqref{1ineq_DP3}, we have
$$
0<\left| k\dfrac{\log \alpha}{\log 10}-n\right|<\dfrac{1.8}{10^{n-1}}.
$$
Dividing this inequality by $k,$ we get
\begin{align}
0<\left| \dfrac{\log \alpha}{\log 10}-\dfrac{n}{k}\right|<\dfrac{1.8}{k\cdot 10^{n-1}}.
\end{align}
Now, we assume that $n\ge 23.$ Then it can be seen that
$$
\dfrac{10^{n-1}}{2(1.8)}>10n\log 10>2.1\times 10^{19}>k,
$$
and so we have
\begin{align}
\left| \dfrac{\log \alpha}{\log 10}-\dfrac{n}{k}\right|<\dfrac{1.8}{k\cdot 10^{n-1}}<\dfrac{1}{2k^2}.
\end{align}
From Lemma~\ref{Lemma-Legendre}, it is seen that the rational number $\dfrac{n}{k}$ is a convergent to $\kappa:=\dfrac{\log\alpha}{\log 10}.$ Now let $\dfrac{p_t}{q_t}$ be $t$-th convergent of the continued fraction of $\kappa.$ Assume now that $\dfrac{n}{k}=\dfrac{p_t}{q_t}$
for some $t.$ Then we have
$$
q_{40}>2.1\times 10^{19}>10n\log 10>k.
$$
Thus $t\in \{0,1, 2,\ldots,39\}.$  By Lemma~\ref{Lemma-Legendre}, we get
$$
\dfrac{1}{(a(M)+2)\cdot k^2}\le \left|\dfrac{\log \alpha}{\log 10}-\dfrac{n}{k}\right|<\dfrac{1.8}{k\cdot 10^{n-1}}.
$$
Since, $a(M)= \max\{a_i: i=0, 1, 2,\ldots,39\}=106,$ we get
$$
\dfrac{1.8}{10^{22}}\ge \dfrac{1.8}{10^{n-1}}>\dfrac{1}{108\cdot k}>\dfrac{1}{108\cdot 2.1\times 10^{19}},
$$
which is a contradiction. Therefore $n\le 22.$ In all cases $\text{I}$ and $\text{II}$ we can consider  
\begin{align}
n\le 27.
\end{align}

So, it remains to check equation \eqref{main_equation2} in the case $g=10$ for $1\le d_1,d_2,d_3\le 9$, $1 \le n \le 27,$ $1 \le k \le 621,$ $1\le \ell\le 57$ and $1\le m \le 58.$ A quick inspection using Maple 
ends the proof of Theorem~\ref{Application_Lucas}.
\end{proof}

\section{Concluding remarks}\label{The end}

In this section, we bring some observations around the Diophantine equations studied in this paper. First, if $\ell=d_1=1,$ then the equations \eqref{main_equation1} and \eqref{main_equation2} become
\begin{align}\label{main_equation1_re}  
F_k= a\dfrac{g^m-1}{g-1}\cdot b\dfrac{g^n-1}{g-1},
\end{align}
and
\begin{align}\label{main_equation2_re} 
L_k=a\dfrac{g^m-1}{g-1}\cdot b\dfrac{g^n-1}{g-1},
\end{align}
where $a, b, k, m$ and $n$ are positive integers such that $1 \le a, b \le g-1$ and $g\ge 2$ with $m\le n$. Note that the two equations above were studied earlier in papers \cite{EK:2019} and \cite{EKS2:2021} where the authors exclusively mention the base $g$ such that $2\le g\le 10.$ However, the method developed in this paper generalizes the results of these authors and allows us to deduce the following results which follow immediately from Theorems~\ref{main_result1} and \ref{main_result2}.
\begin{cor}
Let $g\ge 2$ be an integer.
\begin{enumerate}
\item The Diophantine equation \eqref{main_equation1_re}  has only finitely many solutions in positive integers  $k, a, b, g, m, n$ such that $1\le a, b\le g-1.$
\item The Diophantine equation \eqref{main_equation2_re}  has only finitely many solutions in positive integers  $k, a, b, g, m, n$ such that $1\le a, b\le g-1.$
\end{enumerate} 
\end{cor}
Next, when we take $g=2$ in \eqref{main_equation1} and \eqref{main_equation2}, we get the following result which gives a link between Fibonacci, Lucas and Mersenne numbers.

\begin{cor}
If $F_k$ and $L_k$ are expressible as  products of three Mersenne numbers, then we have 
$$
F_k\in \{1,\;3,\; 21\} \quad \text{and}\quad L_k\in \{1,\;3,\; 7\}.
$$
\end{cor}
\begin{proof}
The proof is similar to those of Theorems~\ref{Application_Fibo} and \ref{Application_Lucas}.
\end{proof}

\section*{Acknowledgements}

The first author is supported by IMSP, Institut de Math\'ematiques et de Sciences Physiques de l'Universit\'e d'Abomey Calavi. The second author is supported by the Croatian Science Fund, grant HRZZ-IP-2018-01-1313. The third author is partially supported by Purdue University Northwest.


\end{document}